\documentclass[a4paper, reqno]{amsart}
\usepackage{amssymb, amsmath, amsthm, enumerate, mathrsfs, mathtools, tikz}
\usetikzlibrary{arrows,calc}

\numberwithin{equation}{section}
\newtheorem{theorem}{Theorem}

\newtheorem{lemma}[theorem]{Lemma}
\newtheorem*{claim}{Claim}

\theoremstyle{definition}
\newtheorem{definition}[theorem]{Definition}
\newtheorem*{acknowledgment}{Acknowledgment}

\newcommand{\R}{\mathbb{R}}

\newcommand{\N}{\mathbb{N}}
\newcommand{\ud}{\mathrm{d}}

\DeclareMathOperator*{\esssup}{ess\,sup}

\begin{document}

\title{Some remarks on the Lipschitz regularity of Radon transforms}

\author{ Jonas Azzam }
\address{Room 4613, James Clerk Maxwell Building, The King's Buildings, Peter Guthrie Tait Road, Edinburgh, EH9 3FD, UK.}
\email{j.azzam@ed.ac.uk}

\author{ Jonathan Hickman }
\address{Ec 414, Department of mathematics, University of Chicago, 5734 S. University Avenue, Chicago, Illinois,  60637, US.}
\email{jehickman@uchicago.edu}

\author{ Sean Li }
\address{Ry 360J, Department of mathematics, University of Chicago, 5734 S. University Avenue, Chicago, Illinois,  60637, US.}
\email{seanli@math.uchicago.edu}

\begin{abstract} A set in the Euclidean plane is constructed whose image under the classical Radon transform is Lipschitz in every direction. It is also shown that, under mild hypotheses, for any such set the function which maps a direction to the corresponding Lipschitz constant cannot be bounded.
\end{abstract}

\maketitle

\section{Introduction}

Let $R$ denote the classical Radon transform associated to the complex of all affine lines in $\R^2$. That is, given $f \in L^1(\R^2)$ let
\begin{equation*}
R_{\omega}f(t) := \int_{\ell_{\omega}(t)} f\,\ud \mathcal{H}^1 \qquad \textrm{for $(\omega, t) \in S^1 \times \R$}
\end{equation*}
where $\ell_{\omega}(t) := \{x \in \R^2 : \langle x, \omega \rangle = t\}$ is the line in the direction of\footnote{Here and throughout this document, for any $\omega \in S^1$ let $\omega^{\perp}$ denote the vector obtained by rotating $\omega$ by $\pi/2$ clockwise about the origin.} $\omega^{\perp}$ at distance $t$ from the origin, noting that each function $R_{\omega}f$ is well-defined in an almost-everywhere sense. The purpose of this note is to explore the degree of regularity $R_{\omega}f$ can enjoy when $f = \chi_E$ is taken to be the characteristic function of some measurable set $E \subseteq \R^2$. In particular, one is interested in examples of sets $E$ with the property that for \emph{all} directions $\omega \in S^1$ the function $R_{\omega}\chi_E$ is Lipschitz. Both a positive and a negative result are established: the former demonstrates a non-trivial example of a set $E$ whose Radon transform is indeed Lipschitz in every direction whilst the latter shows that, under a mild hypothesis on the set, the Lipschitz constant must necessarily be an unbounded function of the direction. To make this discussion precise, for any measurable set $E$ and $\omega \in S^1$ let $\mathrm{Lip}_E(\omega)$ denote the Lipschitz constant of $R_{\omega}\chi_E$, with the understanding that $\mathrm{Lip}_E(\omega) := \infty$ if $R_{\omega}\chi_E(t) = \infty$ for any $t \in \R$.


\begin{theorem}\label{positive result} There exists a measurable set $E \subseteq \R^2$ with $0 < |E| < \infty$ such that $\mathrm{Lip}_{E}(\omega) < \infty$ for all $\omega \in S^1$. 
\end{theorem}

Such a set $E$ is explicitly constructed in this paper. The construction is not a bounded set, but it does satisfy the following weaker property:
\begin{equation}\label{property}
\omega \mapsto \mathcal{H}^1\{t \in \R : R_{\omega}\chi_{\bar{E}}(t)  \neq 0 \} \quad \textrm{is bounded on $S^1$.}
\end{equation}
Moreover, $\mathrm{Lip}_{E} \colon S^1 \to \R$ is an unbounded function, but it transpires that this is necessary whenever $E$ satisfies \eqref{property}.  

\begin{theorem}\label{negative result}
Suppose $E\subseteq \R^{2}$ is measurable and satisfies \eqref{property}. Then $\mathrm{Lip}_{E} \notin L^{2}(S^1)$ and therefore cannot be bounded.
\end{theorem}

These theorems can be viewed as addressing a global variant of an (open) problem raised by Marianna Cs\"ornyei. Given $f \in L^1(\R^2)$ and real numbers $a < b$ consider the truncated Radon transform
\begin{equation*}
R^{a,b}_{\omega}f(t) := \int_{\ell^{a,b}_{\omega}(t)} f\,\ud \mathcal{H}^1 \qquad \textrm{for $(\omega, t) \in S^1 \times \R$}
\end{equation*}
where now $\ell^{a,b}_{\omega}(t) := \{ x \in \R^2 : a < \langle x, \omega^{\perp}\rangle < b \textrm{ and } \langle x, \omega\rangle = t\}$ is a line segment. Cs\"ornyei's question asks whether there exists some measurable $E \subseteq \R^2$ of finite, non-zero measure such that for every choice of $a,b$ and $\omega$ the function $R^{a,b}_{\omega}\chi_E$ is Lipschitz. The results and construction of this paper do not offer any direct progress on this problem. 

The analogous questions are easy in $\R^n$ if $n \geq 3$.  Indeed, the cross sectional volume of a ball in $\R^n$ in any direction behaves like $(r^2 - t^2)^{(n-1)/2}$, which is Lipschitz (and thus uniformly Lipschitz over $\omega \in S^{n-1}$) if $n \geq 3$.  This holds even for the local version of the problem discussed in the previous paragraph.

\begin{acknowledgment} The authors would like to thank Marianna Cs\"ornyei for suggesting the problem.  S.L. was supported by NSF grant DMS-1600804.
\end{acknowledgment} 

\section{Proofs}

To begin the proof of Theorem \ref{negative result} is presented, which is a concise Fourier analytical argument.

\begin{proof}[Proof (of Theorem \ref{negative result})]
Let $\omega \in S^1$ and suppose 
\begin{equation}\label{finite Lipschitz}
\mathrm{Lip}_E(\omega) < \infty.
\end{equation}
 By Rademacher's theorem,
\begin{equation*}
\esssup_{t\in \R} \big|\partial_t R_{\omega} \chi_E(t)\big| \leq \mathrm{Lip}_E(\omega)
\end{equation*}
and hypothesis \eqref{property} implies that $\partial_t R_{\omega} \chi_E(t)$ is non-zero only for $t$ belonging to a set of $\mathcal{H}^1$-measure at most $M > 0$, where $M$ is independent of $\omega$. Thus, if one assumes $\mathrm{Lip}_E \in L^2(S^1)$ so that \eqref{finite Lipschitz} holds for almost every $\omega \in S^1$, then
\begin{equation*}
\int_{S^1}\int_{\R}\big|\partial_t R_{\omega} \chi_E(t)\big|^{2}\,\ud t \ud\omega \leq M \|\mathrm{Lip}_E\|_{L^{2}(S^1)}^{2}.
\end{equation*}
On the other hand, applying the Fourier transform in the $t$ variable, applying the Fourier slice theorem and changing from polar to Cartesian coordinates establishes the well-known identity
\begin{equation*}
\big\|\partial_tR\chi_E\big\|_{L^2(S^1 \times \R)}^2 = 8\pi^2\|\chi_E\|_{\dot{H}^{\frac{1}{2}}(\R^2)}^2,
\end{equation*}
where the right-hand expression involves the homogeneous Sobolev norm. Hence,
\begin{equation*}
\|\chi_E\|_{\dot{H}^{\frac{1}{2}}(\R^2)}^{2} \leq \frac{M}{8\pi^{2}}\|\mathrm{Lip}_E\|_{L^{2}(S^1)}^{2}<\infty.
\end{equation*}
If $f \in \dot{H}(\R^2)$ and $f^{*}$ is the symmetric decreasing rearrangement of $f$, then a classical Sobolev space rearrangement inequality (see, for instance, \cite[Lemma 7.17]{Lieb1997}) states that
\begin{equation*}
\|f\|_{\dot{H}^{\frac{1}{2}}(\R^2)}^{2}  \geq \|f^{*}\|_{\dot{H}^{\frac{1}{2}}(\R^2)}^{2}.
\end{equation*}
Furthermore, recall there is a constant $c$ such that the integral formula
\begin{equation*}
\|f\|_{\dot{H}^{\frac{1}{2}}(\R^2)}^{2} = c\iint_{\R^2 \times \R^2} \frac{|f(x)-f(y)|^{2}}{|x-y|^{3}}\,\ud x\ud y
\end{equation*}
holds for any $f \in \dot{H}^{\frac{1}{2}}(\R^2)$ (see, for instance,  \cite[Theorem 7.12]{Lieb1997}). Combining these facts, if $B$ is the ball centered at $0$ such that $|B|=|E|$, then $\chi_E^{*}=\chi_{B}$ and so
\begin{align*}
\|\chi_E\|_{\dot{H}^{\frac{1}{2}}(\R^2)}^{2} \geq c\iint_{\R^2 \times \R^2} \frac{|\chi_{B}(x)-\chi_{B}(y)|^{2}}{|x-y|^{3}}\,\ud x\ud y.
\end{align*}
However, it is not hard to show that this last integral is infinite, and one obtains a contradiction. 
\end{proof}

The proof of the Theorem \ref{positive result} follows some simple observations concerning configurations of triangles in the plane.

\begin{definition}  A \emph{standard triangle} is a closed equilateral triangle $T \subseteq \R^2$ with the property that one edge is parallel to the $x$-axis. The common side length of such a triangle is denoted by $\ell(T)$. 
\end{definition}

Letting $\omega_k := (\cos  k \pi /3, \sin  k \pi /3) \in S^1$ for $k = 0,1,2$, it follows that the set of (tangent) directions of the edges of a standard triangle is $\{\pm\omega_0, \pm\omega_1, \pm\omega_2\}$. One therefore immediately observes that if $T$ is a standard triangle and $\omega \in \{\pm\omega_0^{\perp}, \pm\omega_1^{\perp}, \pm\omega_2^{\perp}\}$, then $R_{\omega}\chi_T$ is discontinuous. In particular, for each such direction $R_{\omega}\chi_T$ admits a single jump discontinuity of height $\ell(T)$. Away from these directions, however, the mappings behave well and it is useful to record the following elementary geometric observation.

\begin{lemma}\label{triangle Lipschitz} Let $T$ be any standard triangle.  If $\omega \in S^1 \setminus \{\pm\omega_0^{\perp}, \pm\omega_1^{\perp}, \pm\omega_2^{\perp}\}$ then $R_{\omega}\chi_T$ is piecewise linear and Lipschitz. Furthermore, the Lipschitz constant depends only on the direction $\omega$ and not on the choice of underlying standard triangle. 
\end{lemma} 

The construction of the set $E$ proceeds by taking a standard triangle $T_1$ and modifying it so as to ameliorate the discontinuities. 

\begin{definition}[Feet] Given a standard triangle $T$ and $r > 0$ define the \emph{$r$-feet} of $T$ to be the three standard triangles of side-length $r$ formed by extending the edges of $T$ by $r$ on each side, as demonstrated in Figure \ref{feet figure}. Given an $r$-foot $\tau$ of $T$ its \emph{outer edge}, denoted $\mathrm{out}(\tau)$, is the edge which lies opposite the common vertex of $\tau$ and $T$.
\end{definition} 

\begin{figure}
\centering
\begin{tikzpicture}[
        scale=2.5,
        IS/.style={blue, thick},
        LM/.style={red, thick},
        axis/.style={very thick, ->, >=stealth', line join=miter},
        important line/.style={thick}, dashed line/.style={dashed, thin},
       dotted line/.style={dotted, thin},
        every node/.style={color=black},
        dot/.style={circle,fill=black,minimum size=4pt,inner sep=0pt,
            outer sep=-1pt},
    ]
   

\path[] (0.5,0)
        node[dot] (v1) {};
\path[] (1.5,0)
        node[dot] (v2) {};
\path[] (1,0.866)
        node[dot] (v3) {};

\path[] (1,0.4)
        node[]  {$T$};


\path[] (1.75,-0.433)
        node[dot] (x1) {};
\path[] (2,0)
        node[dot] (x2) {};
 \draw[important line]
            (x1) coordinate (es) -- (x2) coordinate (ee);

\path[] (1.75,-0.2)
        node[]  {$\tau_1$};


\path[] (1.25,1.299)
        node[dot] (y1) {};
					
\path[] (0.75,1.299)
        node[dot] (y2) {};
 \draw[important line]
            (y1) coordinate (es) -- (y2) coordinate (ee);

\path[] (1, 0.866+0.232)
        node[]  {$\tau_0$};

				
\path[] (0,0)
        node[dot] (z1) {};				
\path[] (0.25,-0.433)
        node[dot] (z2) {};
 \draw[important line]
            (z1) coordinate (es) -- (z2) coordinate (ee);

\path[] (0.25,-0.2)
        node[]  {$\tau_2$};


\draw[important line]
            (x1) coordinate (es) -- (y2) coordinate (ee);
\draw[important line]
            (y1) coordinate (es) -- (z2) coordinate (ee);
\draw[important line]
            (z1) coordinate (es) -- (x2) coordinate (ee);
\end{tikzpicture}
\caption{A standard triangle $T$ and its $r$-feet $\tau_0, \tau_1, \tau_2$.} \label{feet figure}
\end{figure}
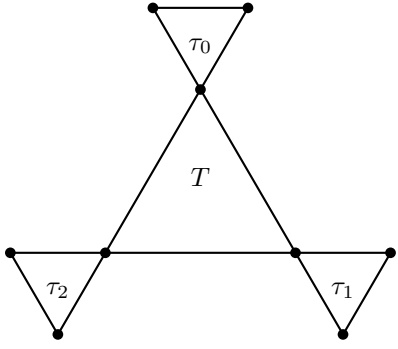

\begin{definition}[Cells] Given a standard triangle $T$, the \emph{cell} $\mathrm{Cell}(T)$ is the subset of $T$ defined by
\begin{equation*}
\mathrm{Cell}(T) := T \setminus \Big(\mathrm{int}(T^0) \cup \bigcup_{k=0}^2 \tau_k^0 \Big)
\end{equation*}
where $T^0$ is the unique standard triangle whose $\ell(T^0)/2$-feet $\tau_0^0, \tau_1^0, \tau_2^0$ have outer edges contained in the edges of $T$. It is easy to see $T^0$ is concentric to $T$ and $\ell(T^0) = (2/7)\cdot\ell(T)$. 
\end{definition} 

\begin{proof}[Proof of Theorem \ref{positive result}] A sequence of concentric standard triangles $\{T_j\}_{j=1}^{\infty}$will be recursively constructed: the set $E$ will then be defined in terms of the $T_j$ and feet and cells associated to these triangles. Let $T_1$ denote a standard triangle centred at 0 with $\ell(T_1) = 1$ and suppose $T_1, \dots, T_{j}$ have all been constructed for some $j \geq 1$. 
\begin{itemize}
\item If $j$ is odd, then let $\{\tau_{j,k}\}_{k=0}^2$ denote the three $(6/7)\cdot 2^{-j}$-feet of $T_j$.
\item If $j$ is even, then let $\{\tau_{j,k}\}_{k=0}^2$ denote the three $2^{-j}$-feet of $T_j$.
\end{itemize}
In either case define $T_{j+1}$ to be the unique triangle whose edges contain the sets $\mathrm{out}(\tau_{j,k})$ for $k=0,1,2$. Here the labeling of the feet is chosen so that the line through the origin in the direction $\omega_k^{\perp}$ bisects each $\tau_{j,k}$ for $k = 0,1,2$.

Define $E$ to be the set
\begin{equation}\label{E definition}
(\mathrm{int}\,T_1 \cap \mathrm{Cell}(T_1)) \cup \bigcup_{\substack{j \in \N \\ \mathrm{odd}}} \bigcup_{k=0}^2 \tau_{j,k} \setminus \mathrm{out}(\tau_{j,k})  \cup \bigcup_{\substack{j \in \N \\ \mathrm{even}}} \bigcup_{k=0}^2 \mathrm{Cell}(\tau_{j,k}) \setminus \mathrm{out}(\tau_{j,k}) 
\end{equation}
and $\mathcal{T} := \{T_1\} \cup \{ \tau_{j,k} : j \in \N,\, k =0,1,2 \}$ and note that these objects satisfy the following basic properties.
\begin{enumerate}[i)]
\item The triangles belonging to $\mathcal{T}$ are mutually disjoint and therefore
\begin{equation*}
|E| < |T_1| + 3\sum_{j=1}^{\infty} |\tau_{j,1}| < \frac{\sqrt{3}}{4}\big(1 + 3\sum_{j=1}^{\infty} 2^{-2j}\big) = \frac{\sqrt{3}}{2}.
\end{equation*}
\item Given any affine line $\ell \subset \R^2$ one may crudely estimate
\begin{equation*}
\mathcal{H}^1(\ell \cap \bar{E}) \leq \mathcal{H}^1(\ell \cap T_1) + 3\sum_{j=1}^{\infty} \mathcal{H}^1(\ell \cap \tau_{j,1}) < \frac{\sqrt{3}}{2}\big(1 + 3\sum_{j=1}^{\infty} 2^{-j}\big)= 2\sqrt{3}.
\end{equation*}
\item The triangles belonging to $\mathcal{T}$ are well-spaced and, in particular, for any $0 \leq k \leq 2$ one may readily deduce from the construction that 
 \begin{equation}\label{spacing}
\min \{ \mathrm{dist}(\tau_{j,k}, \tau_{i,k}) : 0 \leq i \leq j-1\} \geq \left\{ \begin{array}{ll} 
2^{j-2} & \textrm{if $j \geq 2$} \\
0 & \textrm{if $j = 1$}
\end{array}\right. ,
\end{equation}
where $\tau_{0,k} := T_1$. 
\end{enumerate}

It remains to verify that the $R_{\omega} \chi_E$ satisfy the property described in Theorem \ref{positive result}. The first step is to show that most lines intersect few of the constituent triangles of $E$. 

\begin{claim} Let $\omega \in S^1 \setminus \{\pm\omega_0, \pm\omega_1, \pm\omega_2\}$ and suppose $\ell$ is a line with direction $\omega^{\perp}$. Then $\ell$ intersects at most $O_{\omega}(1)$ of the triangles belonging to $\mathcal{T}$. 
\end{claim}

\begin{proof}  By the rotational symmetry of $E$ it suffices to prove the lemma with $\mathcal{T}$ replaced by $\mathcal{T}_1 :=  \{ \tau_{j,1} : j \in \N\}$. 

From the choice of direction, there exists some $N \in \N$ such that 
\begin{equation*}
2^{-N} < \angle(\omega^{\perp}, \omega_1^{\perp}) \leq 2^{-N+1}
\end{equation*}
where $0 \leq \angle(v, w) \leq \pi/2$ denotes the (unsigned) acute angle between the directions $v,w \in S^1$.
 
Suppose that there exist at least $M := \max \{N,6\}$ triangles belonging to $\mathcal{T}_1$ which intersect $\ell$. For notational simplicity let $\tau_0 := T_1$ and $\tau_{j} := \tau_{j,1}$ for all $j \in \N$. Then there exist some $j_1, j_2 \geq M-2$ with $j_1 > j_2$ such that $\ell \cap \tau_{j_1}$ and $\ell \cap \tau_{j_2}$ are non-empty.
Moreover, by the definition of $\mathcal{T}_1$ and \eqref{spacing} it follows that there exist $x_1, x_2 \in \ell$ which satisfy $\mathrm{dist}(x_1, x_2) \geq 2^{M-4}$.  We also have that $x_1,x_2$ lie at a distance at most $2^{-M+1}$ from $\mathrm{span}\{\omega_1^{\perp}\}$ as $\tau_j = \tau_{j,1}$ lies in a $2^{-j+1}$ neighborhood of $\omega_1^\perp$. Consequently, 
\begin{equation*}
\angle(\omega^{\perp}, \omega_1^{\perp}) \leq \tan \angle(\omega^{\perp}, \omega_1^{\perp}) \leq 2^{-2M+6} \leq 2^{-N}
\end{equation*}
which contradicts the definition of $N$. Thus, there can be at most $\max\{N, 6\}$ triangles belonging to $\mathcal{T}_1$ which intersect $\ell$, giving the desired bound. 
\end{proof}

Suppose $\omega \in S^1 \setminus \{\pm\omega_0, \pm\omega_0^{\perp}, \pm\omega_1, \pm\omega_1^{\perp}, \pm\omega_2, \pm\omega_2^{\perp}\}$ and write $R_{\omega} \chi_E$ as a sum of $R_{\omega} \chi_{\tau}$ as $\tau$ varies over all sets appearing in the essentially disjoint union on the right-hand side of \eqref{E definition}. By Lemma \ref{triangle Lipschitz}, each of the functions appearing in this sum is piecewise linear and Lipschitz and the resulting Lipschitz constants are uniformly bounded. Furthermore, the above claim ensures that for each $t \in \R$ the values $R_{\omega}\chi_{\tau}(t)$ are non-zero for only $O_{\omega}(1)$ choices of $\tau$. Combining these observations one deduces that $R_{\omega} \chi_{E}$ is itself a Lipschitz function.

Now suppose $\omega \in \{\pm\omega_0^{\perp}, \pm\omega_1^{\perp}, \pm\omega_2^{\perp}\}$. Here one can immediately see from the construction of the set that the alignment and relative proportions of the feet ensure that $R_{\omega} \chi_E$ is a Lipschitz, countably-piecewise linear function.  The idea is that for $k \in \{0,1,2\}$, the edge of $\tau_{j,k}$ (resp. $\mathrm{Cell}(T_1)$) orthogonal to the direction $\omega_k^\perp$ is matched by the edges of $\tau_{j+1,k+1}$ and $\tau_{j+1,k+2}$ (resp. $\tau_{1,k+1}$ and $\tau_{1,k+2}$), where the addition is taken modulo 3, in a way that makes the resulting Radon transform Lipschitz.

It remains to consider the case $\omega \in \{\pm\omega_0, \pm\omega_1, \pm\omega_2\}$. By rotational symmetry one may assume without loss of generality that $\omega = \omega_0$. It is clear that $F := R_{\omega} \chi_E$ is Lipschitz and bounded on the restricted domain $\{t \in \R : |t| \geq 1/4\}$ and it therefore suffices consider the behaviour of $F$ on $[-1/2,1/2]$. Since $F$ is an even function the problem further reduces to showing $F$ is Lipschitz on $[-1/2, 0]$. 

Letting
\begin{equation*}
E_1 := (\mathrm{int}\,T_1 \cap \mathrm{Cell}(T_1)) \cup (\tau_{1,1} \setminus \mathrm{out}(\tau_{1,1}))
\end{equation*}
and $f_j(t):= R_{\omega}\chi_{2^{-2j}E_1}(t)$ one may easily observe that
\begin{equation*}
F(t) = \sum_{j=0}^{\infty} f_j(t) 
\end{equation*}
for all $t \in [-1/2, 0]$. This follows from the fact that when $j$ is even $\tau_{j,0}$ and $\tau_{j+1,0}$ together behave like a $2^{-j}$-scaled copy of $E_1$ with respect to Radon transforms in the direction $\omega_0$.  Furthermore, for $i=1,\dots, 5$ define $I^i_j := 2^{-2j}I^i$ where
\begin{gather*}
I^1 := [-1/2, -2/7], \quad I^2 := [-2/7, -3/14], \quad I^3 := [-3/14, -1/7], \\
I^4:= [-1/7, -1/14], \quad I^5 := [-1/14, 0]
\end{gather*}
and, in addition, let $I_j^0 := [-1/2, -2^{-2j}/2]$. Then for $j \in \N_0$ it follows that 
\begin{equation*}
f_j(t) := \sqrt{3} \times \left\{\begin{array}{ll}
0 & \textrm{if $t \in I^0_j$} \\
t + 2^{-2j-1} & \textrm{if $t \in I^1_j$} \\
(3/7)\cdot 2^{-2j-1} & \textrm{if $t \in I^2_j$} \\
3t + (3/7)\cdot2^{-2j+1} & \textrm{if $t \in I^3_j$} \\
t + (1/7)\cdot2^{-2j+2} & \textrm{if $t \in I^4_j$} \\
2^{-2j-1} & \textrm{if $t \in I^5_j$} 
\end{array} \right. .
\end{equation*}

Letting $F_k(t) := \sum_{j=0}^k f_j(t)$ suppose $k \geq 2$ and note the following.
\begin{itemize}
\item For $t \in I^0_k$ one has $f_k(t) = 0$ and it follows that $F_k(t) = F_{k-1}(t)$.
\item For $t \in I^5_{k-1}$ and $0 \leq j \leq k-1$ one has $f_j(t) = \sqrt{3}\cdot 2^{-2j-1}$ and therefore $F_k(t) = f_k(t) + C_{k-1}$ where $C_{k-1} := (2/\sqrt{3})\cdot(1 - 2^{-2k})$.
\item Fixing $t \in I^1_{k}$ and writing $F_k(t) = F_{k-2}(t) + f_{k-1}(t) + f_{k-2}(t)$, since  $I^1_k \subset I^5_{k-2}$  one may deduce that
\begin{equation*}
F_k(t) = C_{k-2} + f_{k-1}(t) + f_{k}(t).
\end{equation*}
\end{itemize}
Note that $[-1/2, 0] =  I^0_k \cup I^1_{k} \cup I^5_{k-1}$ and so the above analysis determines the value of $F_k$ on the whole interval of interest.

It now follows by induction that each $F_k$ is a piecewise linear function on $[-1/2, 0]$ with derivative (where defined) bounded above by $3\sqrt{3}$ and, consequently, $F$ is Lipschitz.

\end{proof}

%

\bibliography{Reference}
\bibliographystyle{amsplain}

\end{document}